\documentclass[12pt]{article}
\usepackage{amsmath}
\usepackage{amsfonts}
\usepackage{amssymb}
\usepackage{graphicx}
\usepackage{setspace}
\usepackage{soul}
\usepackage{xcolor}
\usepackage[numbers,sort&compress]{natbib}
\RequirePackage[colorlinks=black,citecolor=blue,urlcolor=blue]{hyperref}
\newtheorem{theorem}{Theorem}[section]
\newtheorem{lemma}{Lemma}[section]
\newtheorem{corollary}{Corollary}[section]

\newtheorem{definition}{Definition}[section]
\newtheorem{assump}{Assumptions}[section]
\newenvironment{proof}{\smallskip\noindent{\it Proof}}{$\Box$}
\numberwithin{equation}{section}
\newcommand{\bpi}{{\boldsymbol \pi}}

\newcommand{\bxi}{{\boldsymbol \xi}}

\def\by{{\boldsymbol y}}
\def\bA{{\boldsymbol A}}

\def\bU{{\boldsymbol U}}
\def\bV{{\boldsymbol V}}

\def\bY{{\boldsymbol Y}}

\def\bc{{\boldsymbol c}}

\def\bv{{\boldsymbol v}}

\def\bp{{\boldsymbol p}}

\def\bx{{\boldsymbol x}}
\def\bX{{\boldsymbol X}}

\def\cM{{\mathcal M}}
\def\cB{{\mathcal B}}

\def\cF{{\mathcal F}}
\def\cI{{\mathcal I}}

\def\cG{{\mathcal G}}
\def\bbP{{\mathbb P}}

\def\bbR{{\mathbb R}}

\graphicspath{%
    {converted_graphics/}% inserted by PCTeX
    {/}% inserted by PCTeX
}
\begin{document}

\title{From generalized arithmetic means to geodesics to Hamilton dynamics
to Bregman divergences} 
\author{Henryk Gzyl,\\
\noindent 
Centro de Finanzas IESA, Caracas, Venezuela.\\
 henryk.gzyl@iesa.edu.ve}

\date{}
 \maketitle

\setlength{\textwidth}{4in}
\vskip 1 truecm
\baselineskip=1.5 \baselineskip \setlength{\textwidth}{6in}
%\newpage
\begin{abstract}
Here we examine some connections between the notions of generalized arithmetic means, geodesics, Lagrange-Hamilton dynamics and Bregman divergences. In a previous paper we developed a predictive interpretation of generalized arithmetic means. That work was more probabilistically oriented.  Here we take a geometric turn, and see that generalized arithmetic means actually minimize a geodesic distance on $\bbR^n.$ Such metrics might result from pull-backs of the Euclidean metric in $\bbR^n.$ 
We shall furthermore see that in some cases these pull-backs might coincide with the Hessian of a convex function. This occurs when the Hessian of a convex function has a square root that is the Jacobian of a diffeomorphism in $\bbR^n.$ In this case we obtain a comparison between the Bregman divergence defined by the convex function and the geodesic distance in the metric defined by its Hessian. 
\end{abstract}


\begin{spacing}{0.05}
   \tableofcontents
\end{spacing}

\section{Introduction and Preliminaries}
In \cite{GH2} a reinterpretation of arithmetic means as best predictors was presented. To explain, consider a strictly increasing function $u:\cI \to \bbR,$ where $\cI\subset \bbR$ is an open interval. If $\{x_1,...,x_n\}$ is a collection of points in $\cI,$ their generalized arithmetic mean 
\begin{equation}\label{AM}
 c = u^{-1}\Big(\frac{1}{n}\sum_{i=1}^n u(x_i)\Big)
\end{equation}
is the point at which the distance $\Big(\sum(u(x_i)-u(c))^2\Big)^{1/2}$ is minimal. This approach leads to a notion of best predictor for random variables with extends the notion of expected value and conditional expected value. At this point it is appropriate to mention that this variational argument to obtain ({\ref{AM}) can be traced back at least to \cite{BC}. From a historical point of view, the generalized arithmetic mean abstracts the notion of certain equivalent of a random cash flow for an investor with utility function $u.$  

\subsection{Organization of the paper}
The aim of this work is to look into the geometric ideas behind that construction. For that we shall use $u(x)$ to define an appropriate metric on $\cM$ and verify that the distance 
\begin{equation}\label{dist0}
d^2_u(\bx,\by)=\Big(\sum(u(x_i)-u(y_i))^2\Big)^{1/2} 
\end{equation}
is actually the geodesic distance between $\bx$ and $\by.$ This simple case is nice because it leads to a diagonal metric in $\cM$ and the realization of the metric as the Hessian of a convex function is trivial. This is carried out in Section 2. In Section 3 we consider a more general diffeomorphism $\bU$ defined on $\cM,$ which leads to a metric such that the geodesic distance in that metric determines a multidimensional version of the generalized arithmetic mean.

We devote Section 4 to consider the changes of variables and the integration of the geodesic equations within the framework of Hamilton-Lagrange formalism. The idea is to regard the integration of the equations of the geodesics from another point of view. It is in Section 5 where we establish conditions for the two way connection between Hessian matrices of convex functions and diffeomorphisms in $\bbR^n,$ while in Section 6 we compare the geodesic distance to the Bregman divergence defined by the convex function. In the last section we collect some examples and pending issues.\\
{\bf A word about notation, etc.}\\
In the remainder of this paper we use $\cM=(\cI)^n,$ where as said, $\cI$ is a an interval (bounded or unbounded). We use this setup because in many cases we want to consider convex functions of the type $\sum u(x_i)$ with $u(x)$ convex. As the manifold that we are considering is a simple subset of $\bbR^n$ naturally provided with a global chart, the standard  constructs of differential geometry in this case are very simple. For example, the tangent bundle and the cotangent bundles are trivially identifiable with $\cI\times\bbR^n.$  We use the standard $\langle\bv_1,\bv_2\rangle$ to denote the usual Euclidean scalar product between the vectors $\bv_1, \bv_2.$ And to finish, we use the standard $h'(x)$ to define the derivative of the function $h(x):\cI\to\bbR,$ and $\dot{x}$ and $\ddot{x}$ to define the first two derivatives of $x(t):(a,b)\to\bbR$ where $t$ is thought of as ``time.''

\section{Geodesics and one-dimensional changes of scale}
 Let $u:\cI \to\bbR.$ be a twice continuously differentiable, strictly increasing function and $u'(x)$ its derivative. At each $\bx\in\cI$ define the Riemannian metric (on the tangent space to $\cM$ at $\bx$) by
\begin{equation}\label{RM1}
g_{i,j}(\bx) = g(x_i)\delta_{i,j} = u'(x_i)^2\delta_{i,j}.
\end{equation}
That is, not only is the coordinate system orthogonal (that is, the metric is diagonal), but it is separable as well, that is, $g_{i,i}$ is a function of $x_i$ only. The equation of the geodesic that minimizes the distance between two points $\bx,\by\in\cM$ is obtained minimizing
\begin{equation}\label{dist1}
\int_0^1\Big(\sum_{i,j}g_{i,j}(\bx)\dot{x}_i\dot{x_j})\Big)^{1/2}dt
\end{equation}
\noindent over the class of continuous functions $\bx:[0,1]\to\cM,$ twice continuously differentiable on $(0,1).$  It is a standard result, that in this case, the equations of the geodesics are (see \cite{H}, for example):
\begin{equation}\label{geo1}
\frac{d}{dt}\Big(\frac{\partial}{\partial \dot{x}_k}\frac{1}{2}\sum_{i,j}g_{i,j}(\bx)\dot{x}_i\dot{x_j}\Big) = \frac{\partial}{\partial x_k}\Big(\frac{1}{2}\sum_{i,j}g_{i,j}(\bx)\dot{x}_i\dot{x_j}\Big),\;\;\;k=1,...,n.
\end{equation}
In this special diagonal metric, if we put $g'(x)=dg(x)/dx,$ the equations of the geodesic are
$$g(x)\ddot{x}_k + \frac{1}{2}g'(x_i)(\dot{x}_k)^2,\;\;k=1,...,n.$$
Substituting $g(x_k)=(u'(x_k)$ we obtain
$$u'(x_k)\ddot{x}_k+u''(x_k)\dot{x} = 0,\;\;\;k=1,...,n$$
This can be simply integrated. At the first step, note that it we obtain that $u'(x_k(t))\dot{x}_k = C_k,$ where $C_k$ is some unknown constant. This equation can be trivially integrated to obtain 
$u(x_k(t)) = u(x_k(0)) + tC_k.$ To determine $C_k$ make use of the fact that $x_k(0)=x_k$ and 
$x_k(1)=y_k$ to obtain $C_k=u(y_k)-u(x_k).$ To close the circle, note that the geodesic distance between $\bx$ and $\by$ is, according to (\ref{dist1})
$$\int_0^1\Big(\sum_{i,j}g_{i,j}(\bx)\dot{x}_i\dot{x_j})\Big)^{1/2}dt = \int_0^1\Big(\sum_i C_i^2\Big)^{1/2}dt = \Big(\sum_i (u(y_k) - u(x_k))^2\Big)^{1/2}$$
\noindent which equals $d_u(\bx,\by).$

\section{Geodesics induced by a diffeomorphism}
Now we extend the setup of the previous section to the general case. Let $\bU:\cM \to \bbR^n$ be a twice continuously differentiable diffeomorphism between $\cM$ and $\bU(\cM).$ On the tangent bundle $\bU(\cM)\times\bbR^n$ we consider the Euclidean metric. Its pullback to $\cM\times\bbR^n$ at $\bx\in\cM$ is given by    
\begin{equation}\label{RM2}
g_{i,j}(\bx) = \sum_{k=1}^nU^k_iU^k_j,\;\;for\;\;1\leq i,j \leq n.
\end{equation}
Here, $U^k(\bx)$ denotes the $k-$th component of $\bU$ and $U^k_i$ denotes $\partial U^k/\partial x_j.$ We denote by $V^k_j$ the inverse of $U^k_j,$ that is, the Jacobian of the inverse of $\bU.$ Again, the geodesics minimize the distance given by (\ref{dist1}) with $g_{i,j}$ given by (\ref{RM2}). If $t\to\bx(t)$ is a geodesic from $\bx(0)=\bx$ to $\bx(1)=\by,$ then it must satisfy equation (\ref{geo1}). 

If in (\ref{geo1}) we substitute (\ref{RM2}), after some simple but boring arithmetics we obtain
$$\sum_{i=1}^nU^i_k\Big(\sum_{n=1}^nU^i_n\ddot{x}_n + \sum_{k,l}U^i_{n,l}\dot{x}_n\dot{x}_l\Big) = 0,\;\;\;k=1,...,n.$$
Now, multiply both sides by $V^k_j$ and sum over $k=1,...,n$ to obtain
$$\sum_{n=1}^nU^j_n\ddot{x}_n + \sum_{k,l}U^j_{n,l}\dot{x}_n\dot{x}_l = 0,\;\;\;j=1,...,n.$$
This can be written as
$$\frac{d}{dt}\Big(\sum_{n=1}^nU^j_n\dot{x}_n \Big)  = 0,\;\;\;j=1,...,n.$$
Which implies that
$$\frac{d}{dt}U^j(\bx(t) = \sum_{n=1}^nU^j(\bx(t))\dot{x}_j(t)  = C_j,\;\;\;j=1,...,n$$
\noindent for some constants $C_j, j=1,...,n.$ Again, This implies that $U^j(\bx(t)) = U^j(\bx(0)) + tC_j$ and that $C_j = U^j(\by)-U^j(\bx),\, j=1,...,n.$ As above, inserting this into the definition of geodesic distance between $\bx$ and $\by$ we obtain
\begin{equation}\label{dist2}
\int_0^1\Big(\sum_{i,j}g_{i,j}(\bx)\dot{x}_i\dot{x_j})\Big)^{1/2}dt = \int_0^1\Big(\sum_i C_i^2\Big)^{1/2}dt = \Big(\sum_i (U^k(\by) - U^k(\bx))^2\Big)^{1/2} = d_U(\bx,\by).
\end{equation}
Notice that $d_U(\bx,\by)$ is positive, symmetric, satisfies the triangle inequality, and since $\bU$ 
is a diffeomorphism, $d_U(\bx,\by)=0 \Leftrightarrow \bx=\by,$ thus the notation is consistent. The comparison with the first case is clear: There $U^k(\bx)=u(x_k)$ for $k=1,...,n.$

\section{A detour into classical mechanics}
Here we address the issue of integrating the geodesic equations from the point of view of classical mechanics. This is to explain the trivial integration of the geodesic equations within the context of an elegant framework. We refer the reader to \cite{A} for the essential notions. Cutting some edges, the general approach to Newton's equations, consists of specifying some manifold ($\cM$ in our case), on whose tangent bundle $\cM\times\bbR^n$ a Lagrangian function $L(\bx,\dot{\bx})$ is defined. A trajectory of the system between points $\bx(1)$ and $\bx(2)$ at times $t_1, t_2$ respectively, is the curve that minimizes the (action) integral:
$$\int_{t_1}^{t_2}L(\bx(t),\dot{\bx}(t),t)dt$$
over the class of all (twice continuously) differentiable curves joining the said points at the given times. It can be shown, see \cite{A}, that such trajectory satisfies (the Euler-Lagrange) extension of the Newton's equation of motion:
\begin{equation}\label{LE1}
\frac{d}{dt}\Big(\frac{\partial}{\partial\dot{x}_k}L(\bx(t),\dot{\bx}(t),t)\Big) = \frac{\partial}{\partial x_k}L(\bx(t),\dot{\bx}(t),t).
\end{equation}
If we put 
$$L(\bx(t),\dot{\bx}(t)) = \frac{1}{2}\sum_{i,j}g_{i,j}(\bx)\dot{x}_i\dot{x}_j,$$
we recognize the equations of the geodesics as the Euler-Lagrange equations of some dynamical system. Notice that if we consider the curves $\by(t)=\bU(\bx(t))$ in $\bU(\cM),$ then the Lagrangian in these coordinates becomes
$$L(\by,\dot{\by},t) = \frac{1}{2}\sum_{k=1}^n(\dot{y}_k)^2$$
\noindent and the Euler-Lagrange (newton) equations become $\ddot{y}_k=0,$ namely the equations of straight lines. We already saw that in Section 3, hidden in the notation. Another way of understanding that change of variables is to begin with the Hamiltonian associated to the Lagrangian. This is defined by introducing the momentum variables $p_k$ by 
$$p_k = \frac{\partial}{\partial\dot{x}_k}L(\bx(t),\dot{\bx}(t),t) = \sum_{j=1}^ng_{j,k}(\bx)\dot{x}_j.$$
Under the assumption of invertibility of the Jacobian $U^i_j,$ solving for $\dot{\bx}$ and substituting in the expression for $L$ we obtain:
$$H(\bx,\bp) = \frac{1}{2}\sum_{k,j}g^{k,j}(\bx)p_ip_j.$$
where we use the conventional $g^{k,j}$ to denote the inverse of $g_{k,j}.$ The analogue of the Euler-Lagrange equations are the Hamilton equations for $(\bx,\bp)$ is:
\begin{equation}\label{ham1}
\dot{x}_k = \frac{\partial H}{\partial p_k} \;\;\;\dot{p}_k = -\frac{\partial H}{\partial x_k}.
\end{equation}

To complete explaining the change of variables we need to show that there exists a mapping $(\bx,\bp) \to (\by,\bpi)$ from $\cM\times\bbR^n$ to $\bU(\cM)\times\bbR^n,$ under which the Hamiltonian becomes
$$\tilde{H} = \frac{1}{2}\sum_{k=1}^n(\pi_k)^2,$$
and that the change of variables preserves the form of the Hamiltonian equations of motion (\ref{ham1}). The mapping is:
$$y^i = U^i(\bx)\;\;\;\;\mbox{and}\;\;\;\pi_i = \sum_{k=1}p_jV^j_i;\;\;\;i=1,...,n$$
The condition for this mapping to preserve the form of the equations is that the following holds for the given Poisson ``brackets'' 
$$[y^i,y^j] = [\pi_i,\pi_j] = 0,\;\;\;\mbox{and}\;\;\;[y^i,\pi_j] = \delta_{i,j},\;\;i,j=1,...n$$
\noindent which, for any pair of continuously differentiable functions $f,g$ on $\cM\times\bbR^n$ is given by:
$$[f,g] = \sum_{k=1}^n \frac{\partial f}{\partial x_k}\frac{\partial g}{\partial p_k} - \frac{\partial f}{\partial p_k}\frac{\partial g}{\partial x_k}.$$
As we said, we were to cut a lot of corners. See \cite{A} for full details. That the system is trivially integrable, means that (\ref{ham1}), in the new coordinates becomes
$$\dot{y}^i = \frac{\partial\tilde{H}}{\partial \pi_i} = \pi_i\;\;\mbox{and} \;\;\dot{\pi}_i = -\frac{\partial\tilde{H}}{\partial y^i} = 0,$$
which imply that $\pi_i=\rm{constant}=C_i$ and $y^i = y^i(0) + tC_i.$
That is the constants of integration that determine the solutions to the equations of the geodesics are the constant values of the momenta. For trajectories go from a point $\bx$ to a point $\by$ in a unit of time, the momenta are related to the distance between the two points by $y^i(1)-y^i(0)=C^i.$

\section{Hessians of convex functions and diffeomorphisms in $\bbR^n$}
In the two (sub)sections that come up next, we do two things. First we extend the results in \cite{GH2} that were summed up in Section 2, to a more general convex function, after that we examine in which case a diffeomorphism $\bU$ determines a convex function whose Hessian coincides with the Riemann metric given by (\ref{RM2}).

\subsection{From a convex function to diffeomorphisms}
Let us begin by stating some necessary properties of the convex function.
\begin{assump}\label{A1}
Let us now consider an open, convex subset $\cM$ of $\bbR^n$ and a strictly convex function $\Phi:\cM\to\bbR^n$ satisfying\\ 
{\bf 1} $\Phi$ is at least twice continuously differentiable in all variables.\\
Denote its Hessian matrix by $\Phi'',$ that is, $(\Phi'')_{i,j}=\partial\Phi/\partial x_i\partial x_j.$  The strict convexity of $\Phi$ implies that the Hessian $\Phi''$ is positive definite. Let us denote by $k(\bx)$ its smallest eigenvalue and let us suppose that:\\
{\bf 2} There exists $a>0$ such that $a\leq k(x)$ for all $\bx\in\cM,$ or equivalently
\begin{equation}\label{comp1}
0 < a\langle\bxi,\bxi\rangle \leq k(\bx)\langle\bxi,\bxi\rangle \leq \langle\bxi,\Phi''\bxi\rangle,\;\;\;\forall\;\bxi\in\bbR^n.
\end{equation}
\end{assump}
And we now have:
\begin{theorem}\label{main1}
With the notations introduced above, let us suppose that there exists a continuously differentiable square root $S$ of $\Phi'',$ that is, $\Phi''=S^tS$ such that:\\
{\bf a} $S(\bx)$ is invertible at each $\bx\in\cM.$\\
{\bf b} For every $i=1,...,n$ we have $\partial S_{i,j}/\partial x_k = \partial S_{i,k}/\partial x_j,$ for all $1\leq j,k\leq n.$\\
Fix some $\bx_0\in\cM$ and define
$$U^i(\bx) = \int S_{i,j}(\gamma(s))\dot{\gamma}(s)ds$$
\noindent where $\gamma$ denotes any continuously differentiable trajectory between $\bx_0$ and $\bx.$ Then, the mapping $\bU:\cM \to \bbR^n$ with components $U^i(\bx)$ is well defined (up to a constant), and it is a global diffeomorphism satisfying
$$\Phi_{i,j}^{''} = \sum_{m=1}^n\frac{\partial U^m}{\partial x_i}\frac{\partial U^m}{\partial x_j}$$
 \end{theorem}
Before proving the theorem we need the following result.
\begin{lemma}\label{lem1}
With the notations introduced above, let $T$ denote the inverse $S^{-1}$ of $S.$ Then $\|T\|\leq 1/a.$
\end{lemma}
\begin{proof}
In item (2) of Assumption \ref{A1} replace $\bxi$ by $\bV\bxi$ to obtain
$$a\langle \bV\bxi,\bV\bxi\rangle \leq k(\bx)\langle \bV\bxi,\bV\bxi\rangle \leq \langle \bV\bxi,\Phi''\bV\bxi\rangle = \langle\bxi,\bV^t\Phi''\bV\bxi\rangle = \langle\bxi,\bV^tU^tU\bV\bxi\rangle = \langle\bxi,\bxi\rangle.$$
In other words
$$\|\bV^t\bV\| \leq 1/k(x) \leq 1/a \Rightarrow \|\bV\| < K=(\sqrt{a^{-1}})$$
since $\|\bV\|$ is given by the square root of the largest eigenvalue of $\bV^t\bV.$
\end{proof}

Let us now complete the proof of Theorem \ref{main1}.\\
\begin{proof}$\;$\\
Since $\cM$ is convex, it is simply connected and assumption (b) in the statement implies (via Stokes theorem) that $\bU$ is a well defined mapping on $\cM,$ and that its Jacobian $S$ is non vanishing (by assumption), and thus $\bU$ is a local diffeomorphism. Invoking the previous lemma and Hadamard's theorem, see, for example Theorem 2 in \cite{M}, we conclude that $\bU$ is a diffeomorphism satisfying
$$\Phi_{i,j}^{''} = \sum_{m=1}^n\frac{\partial U^m}{\partial x_i}\frac{\partial U^m}{\partial x_j}$$
which concludes the proof.
 \end{proof}

\subsection{From diffeomorphisms to convex functions}
A convex function has a positive definite Hessian. The next result imposes a condition on the diffeomorphism $\bU$ so that the metric that it defines, as explained in Section 3, can be the Hessian of a convex function.
\begin{theorem}\label{main2}
Let $\cM$ be an open convex subset of $\bbR^n$ and $\bU:\cM\to\bbR^n.$ be a diffeomorphism satisfying
\begin{equation}\label{assump3}
\sum_{m=1}^n\frac{\partial^2 U^m}{\partial x_k\partial x_i}\frac{\partial U^m}{\partial x_j} = \sum_{m=1}^n\frac{\partial^2 U^m}{\partial x_i\partial x_j}\frac{\partial U^m}{\partial x_k}
\end{equation}
Then there exists a strictly convex function $\Phi:\cM \to \bbR^n$ such that
$$\frac{\partial^2 \Phi}{\partial x_i\partial x_j} = \sum_{m=1}^n\frac{\partial U^m}{\partial x_i}\frac{\partial U^m}{\partial x_j}$$
\end{theorem}
\begin{proof}
As before, $g_{i,j}=\sum_{m=1}^n\frac{\partial U^m}{\partial x_i}\frac{\partial U^m}{\partial x_j}$
Fix an $1\leq i \leq n$ and an $\bx(0)\in\cM$ and consider the following curve joining $\bx(0)$ to $\bx$ defined piecewise by increasing one coordinate at a time: At the $k-$th step, move along the $k-$th coordinate axis from $x_k(0)$ to $x_k,$ that is, along the line:
$$(x_1,x_2,...,x_{k-1},\xi_k,x_{k+1}(0),...,x_n(0))\;\;\;\mbox{with}\;\;\;x_k(0)\leq \xi_k\leq x_k.$$
Call this trajectory $\gamma.$ The import of condition(\ref{assump3}) is to make the integral defined below to be independent of the trajectory.
$$A^i(\bx) = A^i(\bx_0) + \sum_{k=1}\int_{x_k(0)}^{x_k}g_{i,k}(\gamma)d\xi_k.$$
\noindent where $A^i(\bx_0)$ are constants of integration.
Now, since $\partial A^i/\partial x_j=g_{i,j}= g_{j,i}=\partial A^j/\partial x_i,$ then the following line integral is also independent of the trajectory. Thus integrating along the same piecewise trajectory we put:
$$\Phi(\bx) = \Phi(\bx(0) + \sum_{k=1}\int_{x_k(0)}^{x_k}A^k(\gamma)d\xi_k.$$
We choose the special trajectory so that the verification that the Hessian of $\Phi$ is $g$ is trivial. Thus we prove the claim.
\end{proof}

\section{A comparison result: Divergence versus geodesic distance}
Bregman divergences are a common measure of discrepancy. They are used to compare how different are two objects that can be described by points in convex subset of some many-dimensional space. The definition goes as follows:
\begin{definition}\label{breg}
Let $\cM\subset \bbR^n.$ be an open convex set $\Phi:\cM \to\bbR$ be a strictly convex, continuously differentiable function and put
\begin{equation}\label{breg2}
\delta^2_\Phi(\bx,\by) = \Phi(\bx) - \Phi(\by) - \langle(\bx-\by),\nabla\Phi(\by)\rangle.
\end{equation}
\end{definition}
{\bf Comment:} Since the right hand is non-negative, and vanishes if and only if $\by=\bx,$ the notation on chosen for the left hand side is consistent, even though in general, it is not a true distance. This is why it is called a discrimination function between $\bx$ and $\by.$

This notion was introduced by Bregman in \cite{Br}, and has been used in a variety of applications. For a short list, consult with \cite{GH2}. Actually, the function $\Phi$ considered there is rather simple: 
$$\Phi(\bx) = \sum_{i=1}^n \phi(x_i)$$
\noindent with $\phi:\cI \to \bbR$ being a convex function defined on the interval $\cI$ and $\cM=\cI^n.$ This is a typical example in many applications. The thrust in \cite{GH2} was to compute the geodesic distance defined by the Riemannian metric defined by the Hessian of $\Phi,$ and to compare it to the pseudo distance $\delta_\Phi$ defined in (\ref{breg2}).

In this section we generalize a comparison result previously obtained in \cite{GH2} for the separable case. Before we state the result, we shall present introduce some notations to unclutter  the typography and carry out a few elementary calculations that form the basis of the proof of the result. 

We shall use $\partial_i$ $\partial^2_{i,j}$ to denote partial (and repeated partial) derivatives with respect to $x_i$ (respectively $x_i$ and $x_j$). We shall use Einstein summation convention. That is, for example $\partial_iU^m\partial_jU^m$ stands for $\sum_{m=1}^n\partial_iU^m\partial_jU^m.$ Next we present three instances of the same computation: Once for functions, once for vector fields and once for matrix valued functions. Just to refer to them when the time comes up.

Again let $\cM\subset\bbR^n$ be an open, convex, connected set and $\Phi:\cM\to\bbR$ be a twice continuously differentiable function. Let $\gamma:[0,1]\to\cM$ be a continuously differentiable curve. We shall denote by $\gamma_k$ the $k-$th component of $\gamma.$ Let $\by,\bx\in\bbR^n$ and suppose that $\gamma(0)=\bx$ and $\gamma(1)=\by.$ Below, anytime that $\gamma(t)$ appears as the argument, say of a function defined on $\cM,$ we shorten it to $t.$

Starting from
$$\Phi(\by) = \Phi(\bx) + \int_0^1\partial_i\Phi(s)\dot{\gamma}_i(t)dt.$$
As above $\dot{\gamma}$ stands for the time derivative of $\gamma.$ Apply this same computation to the function $\partial_i\Phi(\bx)$ that appears under the integral sign to obtain.
$$\Phi(\by) = \Phi(\bx) + \int_0^1\Big(\partial_i\Phi(0) + \int_0^t\partial^2_{k,i}\Phi(s)\dot{\gamma}_k(s)dt\Big)\dot{\gamma}_i(t)dt.$$
Now integrate the first term to obtain $(\by-\bx)_i\partial_i\Phi(\bx).$ To complete, exchange the integration over $t$ with that over $s,$ notice that the integral over $t$ becomes $(\by-\gamma(s))_i$ and the whole identity becomes
\begin{equation}\label{one}
\Phi(\by) = \Phi(\bx)+(\by-\bx)_i\partial_i\Phi(\bx)+\int_0^1\big(\by-\gamma(s)\big)_i\partial^2_{k,i}\Phi(s)\dot{\gamma}_k(s)ds.
\end{equation}
A similar identity holds componentwise for vector valued functions $\bA:\cM\to\bbR^n.$ In this case the analogue of (\ref{one}) becomes
\begin{equation}\label{two}
A_m(\by) = A_m(\bx)+(\by-\bx)_i\partial_iA_m(\bx)+\int_0^1\big(\by-\gamma(u)\big)_i\partial^2_{k,i}A_m(u)\dot{\gamma}_k(u)du.
\end{equation}
Let us now rewrite (\ref{one}) as
\begin{equation}\label{three}
\delta_\Phi^2(\by,\bx) = \Phi(\by)-\Phi(\bx)+(\by-\bx)_i\partial_i\Phi(\bx) = \int_0^1\big(\by-\gamma(s)\big)_i\partial^2_{k,i}\Phi(s)\dot{\gamma}_k(s)ds
\end{equation}
and proceed with the right hand side as follows. Suppose that as in Section 3 that there exists a diffeomorphism $\bU$ such that $\partial^2_{i,k}\Phi=\partial_iU^m\partial_kU^m.$ Now, for a fixed $m$  consider only $(\by-\gamma(s)\big)_i\partial_iU^m.$ According to (\ref{two}) in which $\bx$ is replaced by $\gamma(s),$ this can be rewritten as
$$(U^m(\by)-U^m(s)) - \int_s^1(\by-\gamma(u)\big)_i\partial^2_{i,j}U^m(u)\dot{\gamma}_j(u)du.$$
and reinserted back in (\ref{four}) to obtain
$$\delta_\Phi^2(\by,\bx) = \int_0^1\Big((U^m(\by)-U^m(s)) - \int_s^1(\by-\gamma(u)\big)_i\partial^2_{i,j}U^m(u)\dot{\gamma}_j(u)du\Big)\partial_kU^m\dot{\gamma}_k(s)ds.$$
Therefore, the first term under the outer integral becomes
\begin{equation}\label{four}
\int_0^1\big(U^m(\by) - \gamma(s)\big)\partial_kU^m\dot{\gamma}_k(s)ds = \frac{1}{2}\|U(\by) - U(\bx)\|^2 = \frac{1}{2}d_U^2(\bx,\by)
\end{equation}
To rewrite the second term, notice that $\partial_kU^m\dot{\gamma}_k(s)ds=dU^m(\gamma(s)).$ Now, we shall consider a specific trajectory: $\gamma(t)=\bx+t(\by-\bx).$ With this choice we have
$$(\by-\gamma(u)\big)_i\partial^2_{i,j}U^m(u)\dot{\gamma}_j(u) = (1-u)(\by-\bx)_iU^m_{i,j}(u)(\by-\bx)_j$$
\noindent and therefore, the second integral becomes
\begin{equation}\label{five}
\int_0^1\Big(\int_s^1(1-u)\big((\by-\bx)_i\partial^2_{i,j}U^m(u)(\by-\bx)_j\big)du\big)dU^m.
\end{equation}
We now gather there results in the main result of this section.
\begin{theorem}\label{comp1}
With the notations introduced above, suppose that the convex function $\Phi:\cM\to\bbR$ is at least three times continuously differentiable, that its Hessian can be factored as $\partial^2_{i,k}\Phi=\partial_iU^m\partial_kU^m,$ and that the diffeomorphism $\bU$ is at least twice continuously differentiable and such that the sign of 
$K(\bxi)\equiv \partial^2_{i,j}U^m(\bxi)\partial_kU^m(\bxi)$ is constant over $\bxi\in\cM,$ then 
\begin{eqnarray}
 \delta_\Phi^2(\by,\bx) \leq \frac{1}{2}d_U^2(\bx,\by)\;\;\;\mbox{whenever}\;\;K(\bxi) \geq 0,\\\label{eq1}
\delta_\Phi^2(\by,\bx)   \geq \frac{1}{2}d_U^2(\bx,\by)\;\;\;\mbox{whenever}\;\;K(\bxi) \leq 0.\label{eq2}
\end{eqnarray}
\end{theorem}
The proof is contained in the computations carried out above. The conclusion is obtained after substituting (\ref{four}) and (\ref{five}) into (\ref{one}).

In the most common case in applications, when $\Phi(\bx)=\sum_{k=1}^n\phi(x_i),$ that is in the separable case, then $\phi''(x)=(u'(x))^2>0$, then $u'(x)>0$ and the condition upon the sign of $u''(x)$ is equivalent to a condition upon the sign of $\phi'''(x).$ In this case, the result of Theorem \ref{comp1} was obtained in \cite{GH2}.

\section{Examples}
\subsection{Extended generalized arithmetic means}
Here is a direct extension of the notion of generalized arithmetic mean.
\begin{theorem}\label{AM1}
Let us suppose that $\bU(\cM)$ is convex, and consider a set $\{\bx(1),...,\bx(M)\}$ of points in $cM.$
There is a unique point $\bc$ in $\cM$ which minimizes the $d_U$ distance to the set $\{\bx(1),...,\bx(M)\}.$ It is given by
\begin{equation}\label{AM2}
\bc = \bU^{-1}\Big(\frac{1}{M}\sum_{m=1}^M \bU(\bx(m)))\Big).
\end{equation}
\end{theorem}
The proof is easy and follows the pattern of the simple one-dimensional case: To find $\bxi\in \bU(\cM))$ such that $\sum_{m=1}^Md_U(\bx(m)),\bxi)$ is easy and it is given by $\bxi=\frac{1}{M}\sum_{m=1}^M \bU(\bx(m)).$ Since $\bxi\in \bU(\cM),$ let $\bc=\bU(\bxi)$ and we are through.

\subsection{Means defined by a flow}
With the notations of Section 3 and the previous example in mind, let $\bU(t,\bx)$ be the flow associated to the geodesics determined by the diffeomorphism $\bU,$ and the solution to the geodesic equation such that $\bx(0)=\bx$ and $\dot{\bx}(0)=\bxi.$ That is
\begin{equation}\label{flow}
\bU(t,\bx) = \bU^{-1}\big(\bU(\bx) + t\bxi\big).
\end{equation}
To verify that for any real $t,s$ $\bU(t+s,\bx)=\bU(s,\bU(t,\bx))$ is routine -as long as the solution to the geodesic equations is defined for all times. With the aid of (\ref{flow}) one can construct a family of transition kernels $\{P_t(\bx,A): t\geq 0, A \in\cB(\bbR^n)\},$ where $\cB(\bbR^n)$ denotes the Borel-subsets of $\bbR^n,$ as follows:
\begin{equation}\label{ker}
P_t(\bx,A)=I_{A}(\bU(t,\bx))
\end{equation}
\noindent where $I_A$ stands for the usual indicator function of the set $A.$ To verify that $\{P_t\}$ is indeed a (semi)group is trivial using (\ref{flow}). Skipping a considerable amount of detail, it is intuitive that this semi group defines a Markov process $\{\bX_t: t\geq 0\}$ having $\bbR^n$ as state space, and $\{P_t\}$ as transition semi group. Furthermore, this process is such that for any Borel, measurable function $f:\bbR^n\to\bbR,$ we have
$$E[f\big(\bX_{t+s}\big)|\bX_t] =  f\big(\bU(s,\bX_t)\big).$$
In particular, if $X_i(t)$ denotes de $i-$th coordinate of $\bX_t$ for $i=1,...,n$ we have $E[X_i(t+s)|\bX_t]=U_i(s,\bX_t),$ or in vector notation:
$$E[\bX(t+s)|\bX_t]=\bU(s,\bX_t),$$ 
that is, the current position is the best predictor of the future values of the position.

\subsection{Harmonic means in $\bbR^n$}
As second example, consider the inversion with respect to the unit sphere in $\cM=\bbR^n\setminus\{0\}$
The mapping $\bU(\bx)=\bx/\|\bx\|^2$ is an involution of $\cM.$ If $\{\bx(1),...,\bx(M)\}$ is a finite set of points in $cM,$ then
$$\bc = \bU\Big(\frac{1}{M}\sum_{m=1}^M \frac{\bx(m)}{\|\bx(m)\|^2}\Big)$$
is the point closest to $\{\bx(1),...,\bx(M)\}$ is the $d_U$ distance. Thus we have a variational interpretation of the $n-$dimensional extension of the notion of harmonic mean.

\subsection{Multidimensional best predictors}
Here we extend the situation considered in the second example. Let $(\Omega,\cF,\bbP)$ be a probability space, that is, a set $\Omega,$  a $\sigma-$algebra of subsets of $\Omega,$ and a probability $\bbP.$ By the customary $E_\bbP$ we shall denote expectation with respect to $\bbP.$ All the random variables $\bX$ in this section will take values in $\bU(\cM)$ and be such that $E_\bbP[\bU(\bX)^2]<\infty.$ Let $\cG\subset\cF$ be a sub-$\sigma-$algebra. We have
\begin{theorem}\label{BP1}
Define the $\bU$ distance between any two random variables $\bX,\bY$ by 
$$\delta_U(\bX,\bY) = \Big(E_\bbP[d_U(\bX,\bY)^2]\Big)^{1/2}.$$
Then, there is a unique, $\cG-$measurable, square integrable random variable $\bX_\cG$ which satisfies
$$\bX_\cG = argmin\{\delta_U(\bX,\bY)| \bY\;\; \mbox{measurable with respect to}\;\; \cG\}.$$
It is given by 
$$\bX_\cG = \bU^{-1}\Big(E_\bbP[\bU(X)|\cG]\Big).$$
\end{theorem}
This result was obtained in \cite{GH3}. The aim there was to present generalized arithmetic means as best predictors, that is, as the solutions to the variational problem stated in Theorem \ref{BP1}.
A particular case of this result is contained in the following result.
\begin{theorem}\label{BP2}
Let $\cG=\{\emptyset,\Omega\}$ be the trivial sigma algebra, and Let $\bX$ be a random variable taking finitely many values $\{\bx_1,...,\bx_M\}$ with probabilities $P(\bX=\bx_i)=p_i, i=1,...,M.$ The rest of the notations are as before. The best predictor of $X$ given no information is given by the generalized arithmetic mean. As noted at the outset, it happens to coincide with the notion of certainty equivalent.
\begin{equation}\label{umean}
\langle\bX\rangle = \bU^{-1}\Big(E_\bbP[\bU(X)]\Big) = \bU^{-1}\Big(\sum_{k=1}^M\bU(\bx_i)P_i\Big).
\end{equation}
\end{theorem}
 
\subsection{Generalized arithmetic means and convex functions defined by the gradient of a convex function}
A particular case of the results in sections 3 and 5 is provided by strictly convex (or concave) functions, that is functions whose Hessian is strictly positive (or negative). Suppose $\Psi(\bx)$ is a strictly positive function on $\bbR^n,$ then $\nabla\Psi$ is locally invertible. Let us suppose that its inverse is global. If we consider the metric introduced in (\ref{RM2}), that is,
$$g_{i,j}(\bx)= \Psi_{i,k}(\bx)\Psi_{j,k}(\bx)$$
using Einstein's summation convention, then the generalized mean of the set $\{\bx_1,...\bx_M\}$ that it defines is given by
$$\big(\nabla\Psi\big)^{-1}\Big(\frac{1}{M}\sum_{k=1}^M\nabla\Psi(\bx_k)\Big).$$
If the condition mentioned in Section 5.2 hold, namely:
$$\mbox{For each fixed}\;\;i=1,...,n\;\;\Psi_{i,m,k}\Psi_{j,k} = \Psi_{i,j,k}\Psi_{m,k}$$
then there exists a convex function $\Phi(\bx)$ on $\bbR^n$ such that
$$\Phi_{i,j}(\bx) = \Psi_{i,k}(\bx)\Psi_{j,k}(\bx)$$
to which the comparison results established in Section 6 apply, whenever the condition mentioned in Theorem \ref{comp1} holds, namely:
$$\mbox{The sign of}\;\;\;\Psi_{i,j,m}(\bx)\Psi_{k,m}(\bx)\;\;\mbox{is constant over}\;\; \bbR^n\;\;\mbox{independently of}\;\;i,j,k\;\;.$$

\subsection{Geodesics and distances determined by the Fenchel-Legendre conjugate}
Consider the simplest possible case in which $\Phi:\bbR^n\to\bbR$ has a strictly positive Hessian and that $\nabla\Phi:\bbR^n\to\bbR^n$ is a diffeomorphism whose range is $\bbR^n.$ Recall that:
\begin{definition}\label{conj}
The Fenchel-Lagrange dual $\Phi^*$ is defined by
$$\Phi^*(\bxi) = \sup\{\langle\bxi,\bx\rangle-\Phi(\bx)|\bx\in\bbR^n\}.$$
\end{definition}
The basic properties of this definition can be seen, for example, in Borwein and Lewis \cite{BL}. To begin with, we have:
\begin{lemma}
With the assumptions made at the outset of the section, a simple computation shows that:
$$\Phi(\bxi) = \langle\bxi,\big(\nabla\Phi\big)^{-1}(\bxi)\rangle - \Phi\big(\big(\nabla\Phi\big)^{-1}(\bxi)\big).$$
And more importantly
\begin{equation}\label{comp1}
\nabla\Phi^*(\bxi) = \big(\nabla\Phi\big)^{-1}(\bxi).
\end{equation}
And also (using Einstein's summation convention):
\begin{equation}\label{comp2}
\partial^2_{i,k}\phi^*(\bxi)\partial^2_{k,j}\Phi\big(\big(\nabla\Phi\big)^{-1}(\bxi)\big)=\delta_{i,j}.
\end{equation}
\end{lemma}
The following result is known. See \cite{N} for example. Its proof  drops out of a computation.
\begin{theorem}\label{eqgeo}
With the notations introduced above, let $\{\bx(t): 0\leq t \leq 1\}$ be a geodesic in $\bbR^n$ between the points $\bx_1=\bx(0)$ and $\bx_2=\bx(1)$ with respect to the metric given by the Hessian matrix of $\Phi.$ Let us put $\bxi(t)=\big(\nabla\Phi\big)(\bx(t))$ with $\bxi_1=\big(\nabla\Phi\big)(\bx_1)$ and $\bxi_2=\big(\nabla\Phi\big)(\bx_2).$ Then $\bxi(t)$ is a geodesic between $\bx_1$ and $\bxi_2$ in the metric given by the Hessian of $\Phi^*.$
\end{theorem}
\begin{proof}
To verify the assertion it suffices invoke (\ref{comp1}) and (\ref{comp2}) to verify that
$$\partial^2_{i,j}\phi^*(\bxi)\dot{\bxi}_i(t)\dot{\bxi}_j(t) = \partial^2_{i,j}\phi(\bx)\dot{\bx}_i(t)\dot{\bx}_j(t).$$
This is left for the reader to carry out.
\end{proof}

Let us now verify that if the Hessian of $\Phi$ can be factored as in Section 3, then the Hessian
of $\Phi^*$ can be factored as well.
\begin{theorem}\label{factor}
Let $\Phi$ and $\Phi^*$ be as above. Suppose that there is a diffeomorphism $\bU$ of $\bbR^n$ such that
$\partial^2_{ij}\Phi(\bx) = \partial_iU_k(\bx)\partial_jU_k(\bx).$ Then $\bU^*(\bxi)=\bU(\big(\nabla\Phi\big)^{-1}(\bxi))$ is a factorization of the Hessian of $\Phi^*.$
\end{theorem}
\begin{proof}
Observe that invoking (\ref{comp2}) we obtain:
$$\partial_jU^*_k(\bxi)=\partial_lU_k\big(\big(\nabla\Phi\big)^{-1}(\bxi)\big)\partial^2_ni\Phi*(\bxi).$$
From this and (\ref{comp2}) it follows that
$$\partial_iU^*_k(\bxi)\partial_jU^*_k(\bxi) = \partial_{i,j}\Phi^*(\bxi).$$
Thus concludes the proof.
\end{proof}\\
To finish, let us verify that the distances along the geodesics in the Hessians of $\Phi^*$  and $\Phi$ coincide.
\begin{corollary}\label{comp3}
With the notations introduced above, let $\bxi(t)$ and $\bx(t)$) be geodesics described in Theorem \ref{comp1}. Then
$$\delta_{\Phi^*}(\bxi_1,\bxi_2) = \delta_{\Phi}(\bx_1,\bx_2).$$ 
\end{corollary}
\begin{proof}
We saw in Section 3 that $\delta^2_{\Phi^*}(\bxi_1,\bxi_2) = \|\bU^*(\bxi_1)-\bU^*(\bxi_2)\|^2.$
From the definition of $\bU^*$ and since $\bxi(t)=\nabla\Phi(\bx)$ we obtain
$$\delta^2_{\Phi^*}(\bxi_1,\bxi_2) = \|\bU^*(\bxi_1)-\bU^*(\bxi_2)\|^2 = \|\bU(\bx_1)-\bU(\bx_2)\|^2 = \delta^2_\Phi(\bx_1,\bx_2).$$
And thus the assertion is verified.
\end{proof}

{\bf Acknowledgments:} I would like to thank Soumalya Mukhopadhyay  for bringing  references \cite{BC} up to my attention.

\end{document}